\documentclass{amsart}

\usepackage{amsmath, amsthm}
\usepackage{mathtools, mathabx}
\usepackage{tikz}
\usepackage{tikz-cd}
\usepackage{hyperref}

\allowdisplaybreaks

\newtheorem{theorem}{Theorem}[section]
\newtheorem{corollary}[theorem]{Corollary}

\newtheorem{proposition}[theorem]{Proposition}
\newtheorem{conjecture}[theorem]{Conjecture}
\theoremstyle{definition}
\newtheorem{definition}[theorem]{Definition}
\newtheorem*{acknowledgement}{Acknowledgement}
\newtheorem{example}[theorem]{Example}

\theoremstyle{remark}

\numberwithin{equation}{section}
\numberwithin{figure}{section}

\DeclarePairedDelimiter{\abs}{\lvert}{\rvert}

\DeclarePairedDelimiter{\braket}{\langle}{\rangle}

\newcommand*{\opensymbol}{}
\DeclareMathOperator{\aut}{Aut}

\DeclareMathOperator{\Tr}{Tr}
\DeclareMathOperator{\diag}{diag}
\DeclareMathOperator{\metrics}{Met}
\DeclareMathOperator{\Vol}{Vol}
\newcommand*{\openVol}{\Vol^{\opensymbol}}

\DeclareMathOperator{\RG}{RG}
\newcommand*{\openRG}{\RG^{\opensymbol}}
\newcommand*{\CoverOpenRG}{\widetilde\RG^{\opensymbol}}

\newcommand*{\double}{D}
\newcommand*{\zset}[1]{\underline{#1}}
\newcommand*{\halfedges}{\mathfrak{h}}
\newcommand*{\interiorHalfedges}{\halfedges_{\text{int}}}
\newcommand*{\boundaryHalfedges}{\halfedges_{\partial}}
\newcommand*{\edges}{e}

\newcommand*{\vertices}{v}

\newcommand*{\faces}{f}
\newcommand*{\boundaries}{b}

\newcommand*{\Mbar}{\widebar{\mathcal{M}}}
\newcommand*{\moduli}{\mathcal{M}}
\newcommand*{\openModuli}{\moduli^{\opensymbol}}
\newcommand*{\coverOpenModuli}{\widetilde{\moduli}^{\opensymbol}}
\newcommand*{\torusOpenModuli}{\moduli^{\opensymbol \vec{\gamma}}}
\newcommand*{\compactOpenModuli}{\widebar{\moduli}^{\opensymbol}}

\newcommand*{\graphs}{\mathcal{G}}

\newcommand*{\QQ}{\mathbb{Q}} 
\newcommand*{\RR}{\mathbb{R}} 
 
\newcommand*{\ZZ}{\mathbb{Z}}

\newcommand*{\cC}{\mathcal{C}}

\newcommand*{\cH}{\mathcal{H}}

\begin{document}

\title{Combinatorial models for moduli spaces of open Riemann surfaces}
\author[B. Safnuk]{Brad Safnuk}
\address{
    Department of Mathematics\\
    Central Michigan University\\
    Mount Pleasant, MI 48859, U.S.A.}
\email{brad.safnuk@cmich.edu}

\begin{abstract}
    We present a simplified formulation of open intersection numbers, as an alternative to the theory initiated by Pandharipande, Solomon and Tessler.
    The relevant moduli spaces consist of Riemann surfaces 
    (either with or without boundary) with only interior marked points.
    These spaces have a combinatorial description using a generalization of ribbon graphs, with a straightforward compactification and corresponding intersection theory.
    Crucially, the generating functions for the two different constructions of open intersection numbers are identical.
    In particular, our construction provides a complete proof of the statement that this generating function is a solution of the MKP hierarchy, satisfies $W^{(3)}$-constraints, and additionally proves in the affirmative the $Q$-grading conjecture for distinguishing contributions from surfaces with different numbers of boundary components, as was previously proposed by the author.
 \end{abstract}

\maketitle

\section{Introduction}

In \cite{Pandharipande:2014fk}, Pandharipande, Solomon and Tessler constructed a rigorous intersection theory on the moduli space of the disk, and proved that the generating function for these numbers obey a number of constraint conditions that are direct analogues of the KdV equation and Virasoro constraints for intersection theory on moduli spaces of closed Riemann surfaces.
They conjectured that the full generating function for descendant integrals on moduli spaces of open Riemann surfaces (i.e. including all genera and number of boundary components) satisfies an integrable system that they termed the open KdV hierarchy, as well as corresponding Virasoro constraints.

In proving the consistency of the proposed Virasoro constraints and open KdV equations, Buryak \cite{Buryak:2014kx} embedded the open KdV equations into the Burgers-KdV hierarchy. He proposed that this larger hierarchy would capture the descendant integrals that included $\psi$-classes corresponding to boundary marked points.
Alexandrov \cite{Alexandrov:2014gfa, Alexandrov:2015kq} then proved that the solution of the Burgers-KdV hierarchy consistent with the open intersection number generating function is given by the Kontsevich-Penner matrix model, when $Q=1$:
\begin{equation*}
    \tau_Q = \det(\Lambda)^Q \cC^{-1}_{\Lambda}
    \int_{\cH_N} d\Phi \exp \left(
        -\Tr \Bigl( \frac{\Phi^3}{3!} - \frac{\Lambda^2\Phi}{2}
        + Q \log \Phi \Bigr)
    \right).
\end{equation*}
He showed that this matrix model is a solution of the MKP hierarchy, and additionally satisfies $W^{(3)}$-constraints.

In \cite{2016arXiv160104049S}, we used these $W^{(3)}$-constraints to derive a topological recursion formula, in the spirit of Chekhov, Eynard, and Orantin \cite{MR2222762, Eynard:2007kx}, which reconstructs the generating function. The recursion formula itself suggested a conjectural refinement to the generating function, by incorporating a grading parameter $Q$, which distinguishes contributions to the generating function from surfaces with different numbers of boundary components.

To be more precise, we let $\openModuli_{(g, b), k, l}$ be the moduli space of Riemann surfaces with boundary, having genus $g$, $b$ boundary components, $k$ marked points on the boundary, and $l$ interior marked points, and $\compactOpenModuli_{(g, b), k, l}$ its compactification by nodal surfaces. Chern classes corresponding to cotangent lines at the interior marked points are denoted $\psi_i$, while the ones coming from boundary marked points are denoted $\phi_i$. Assuming that all such constructions can be made rigorous, we define
\begin{equation*}
    F_{(g, b), n}(t_1, t_2, \ldots) = 
    \sum_{k+l = n} \sum_{\vec{d}, \vec{f}}
    \frac{1}{k! l!}
    \int\limits_{\compactOpenModuli_{(g, b), k, l}}
    \prod_{i=1}^{l} \psi_i^{d_i} T_{d_i}
    \prod_{j=1}^{k} \phi_j^{f_j} S_{f_j},
\end{equation*}
where the natural generating function parameters $\{T_d\}$, $\{S_f\}$ are related to the KP times $\{t_j\}$ by
\begin{align*}
    T_d &= (2d+1)!! t_{2d+1}, \\
    S_f &= 2^{f+1}(f+1)! t_{2f+2}.
\end{align*}
Then in \cite{2016arXiv160104049S} we proposed the following
\begin{conjecture}
    \label{conj:Q-grading}
    \begin{equation*}
        \tau_Q(t_1(\Lambda), t_2(\Lambda), \ldots) =
        \exp\Bigl(\sum_{g, b, n}\hbar^{2g + b - 2}Q^b F_{(g, b), n}\Bigr),
    \end{equation*}
    where 
    \begin{equation*}
        t_k(\Lambda) = \frac{1}{k}\Tr \Lambda^{-k}.
    \end{equation*}
\end{conjecture}

We note that not all details of the proof of the open KdV conjecture are available at the time of the writing of the present paper.
Notably missing are a rigorous construction of the compactification, and consideration of the boundary behavior of the line bundles corresponding to the boundary marked points.
However, Solomon and Tessler have announced that complete proofs are forthcoming.
Some important details have appeared \cite{Buryak:2015uq, Tessler:2015ys}, including a combinatorial model for the moduli space.
Since, in the intersection theory, the behavior at the boundary plays a crucial role, their formulas for the generating function include nodal ribbon graphs
(i.e. combinatorial models for nodal open Riemann surfaces).
Buryak and Tessler \cite{Buryak:2015uq} have successfully constructed a matrix model whose asymptotic expansion can be expressed as a sum of these nodal open ribbon graphs,
and is a solution of the open KdV hierarchy.
They are also working on extending their work to include $\phi$-classes corresponding to the boundary marked points.

A natural question that arises is what is the relationship between the two different matrix models used to recreate open intersection numbers?
Moreover, is there some geometric understanding that underpins these different models?
It is this question, as well as providing a proof of the $Q$-grading conjecture, that motivated the present work. As a side benefit, we obtain a conceptually simple alternative geometric model for open intersection numbers, which is of interest in its own right.
In particular, we prove directly that the generating function for our intersection theory is given by the Kontsevich-Penner model.
Hence, it is a solution of the modified KP (MKP) hierarchy (c.f. \cite{Alexandrov:2013fj,MR723457}) satisfies $W^{(3)}$-constraints, and allows for a proof of Conjecture~\ref{conj:Q-grading}. 

The main results of the paper are outlined below.
First, we give the Feynman graph expansion of the Kontsevich-Penner model.
This is written as a sum over ribbon graphs with boundary, and explicitly shows that graphs with $b$ boundary components contribute $Q^b$ to the partition function.

Then, we show that these same graphs appear in a combinatorial model for a simplified variant of the moduli space of open Riemann surfaces.
More precisely, we limit ourselves to the case when there are no marked points on the boundary.
If $\openModuli_{(g, b), n}$ is the moduli space of genus $g$ Riemann surfaces with $b$ boundary components and $n$ (internal) marked points, then in Section~\ref{sect:CombinatorialRepresentations} we construct an equivalence of orbifolds
\begin{equation*}
    \openModuli_{(g,b), n} \times \RR_+^{n} \xrightarrow{\sim} \openRG_{(g,b), n},
\end{equation*}
where $\openRG_{(g, b), n}$ is the complex of ribbon graphs with boundary, as defined in Section~\ref{sect:RibbonGraphs}.

The combinatorial model permits an embedding of the open moduli space into the moduli space of closed Riemann surfaces, and a straightforward compactification is obtained by taking the closure of the image in the usual Deligne-Mumford moduli space of nodal curves.
An intersection theory on this moduli space is obtained by considering for any fixed $\vec{x}\in \RR_+^n$, the integral of
\begin{equation*}
    \frac{1}{d!}\left(
        \sum_{i=1}^{n} \frac{x_i^2}{2}\psi_i + \sum_{j=1}^{b}
        \frac{y_j^2}{2}\psi_{j+n}
    \right)^{d} dy_1 \cdots dy_b
\end{equation*}
over a compact subset of $\Mbar_{g, n+b}\times \RR_{\geq 0}^b$ that depends on $\vec{x}$, but naturally arises from our compactification of $\openModuli_{(g,b), n}$.
Moreover, we show that this volume form has an interpretation as a Duistermaat-Heckman measure, reflecting the intrinsic geometry of $\openModuli_{(g, b), n}$.

Using the Weil-Petersson symplectic form, and adapting a scaling limit argument of Do \cite{2010arXiv1010.4126D}, we pull back the scaled sum of $\psi$-classes to the piecewise defined 2-form on $\openRG_{(g,b), n}$ that was first considered by Kontsevich in his proof of the Witten Conjecture \cite{MR1171758}.
By integrating over the ribbon graph complex, we prove that the generating function of open intersection numbers is given by the Kontsevich-Penner model.

The techniques used in this paper are from hyperbolic and symplectic geometry, and it does seem to be the natural setting in which the construction is carried out.
Nevertheless, it would be interesting to obtain an algebro-geometric interpretation for these results.
As well, the relationship between our moduli space and the one being constructed by Solomon and Tessler seems to be quite mysterious, and worth exploring further.
However, such topics are left for future work.

The paper is organized as follows.
In Section~\ref{sect:RibbonGraphs} we define a notion of open ribbon graph (or ribbon graph with boundary) and construct the resulting ribbon graph complex.
In Section~\ref{sect:FeynmanGraphs} we calculate the Feynman graph expansion of the Kontsevich-Penner model, and show that it can be written as a sum over open ribbon graphs.
In Section~\ref{sect:CombinatorialRepresentations} we prove that the ribbon graph complex constructed in Section~\ref{sect:RibbonGraphs} is equivalent to the moduli space of open Riemann surfaces, and, as a consequence, permits for a straightforward compactification.
Finally, in Section~\ref{sect:IntersectionNumbers} we construct an intersection theory on the moduli space, and show that its generating function is given by the Kontsevich-Penner model.

\begin{acknowledgement}
    The author would like to thank Ran Tessler for helpful discussions.
     During the preparation of this paper, the author received support from the National Science Foundation through grant DMS-1308604. 
\end{acknowledgement}

\section{Ribbon graphs with boundary}
\label{sect:RibbonGraphs}
In this section we define a notion of ribbon graph appropriate for modeling surfaces with boundary.
By \emph{ribbon graph} (either with or without boundary) we mean a graph embedded in a compact surface such that its complement is a disjoint collection of disks, and with the boundary of the surface contained in the image of the graph.
We further require the interior of each edge to be either contained in the boundary or disjoint from the boundary.
As well, we disallow graphs with vertices of degree one or two.
The \emph{boundary} of the ribbon graph is the collection of edges and vertices contained in the boundary of the surface.
We will use the terms open or closed ribbon graphs to refer to graphs with or without boundary, respectively. 

The path of edges surrounding a disk in the graph complement of the surface is a \emph{face cycle}, while the path of edges around a boundary component is a \emph{boundary cycle}.
Note that in the literature describing closed ribbon graphs, what we call a face cycle is usually called a boundary cycle.
We have changed standard notation to avoid confusing the different notions of boundary.
The \emph{type} of a ribbon graph is the triple of integers $((g, b), n)$, where $g$ is the genus of the underlying surface, $b$ is the number of boundary components of the surface (equivalently, the number of boundary cycles on the graph), and $n$ is the number of disks in the graph complement (equivalently, the number of face cycles).
A ribbon graph is open if and only if $b > 0$.

An equivalent, purely combinatorial definition of a ribbon graph is given as follows.

\begin{definition}
    \label{defn:RibbonGraph}
    A \emph{ribbon graph} is a collection of data
    $\Gamma = (\halfedges, \sigma_0, \sigma_1, \sigma_2, B)$, where
    \begin{enumerate}
        \item $\halfedges$ is a finite set (the set of directed edges, or equivalently half-edges of the graph).
        \item $\sigma_i: \halfedges \rightarrow \halfedges$ are permutations.
        \item $\sigma_1$ is a fixed-point free involution.
        \item $\sigma_2 =\sigma_0^{-1}\circ\sigma_1$, hence is usually omitted from the notation.
        \item $B \subset \halfedges$ is the set of boundary half-edges, satisfying the conditions:
            \begin{itemize}
                \item[(i)] if $x\in B$ then $\sigma_2(x)\in B$, and
                \item[(ii)]  if $x\in B$ then $\sigma_1(x) \notin B$.
            \end{itemize}
        \item No cycles of $\sigma_0$ are of length 1 or 2.
    \end{enumerate}
\end{definition}

    A \emph{connected} ribbon graph satisfies the additional constraint that the group generated by $\braket{\sigma_0, \sigma_1}$ acts transitively on $\halfedges$.
    A graph satisfying conditions (1)--(5), but having at least one vertex of degree less than three will be called \emph{unreduced}.

    For any ribbon graph $\Gamma$, we can identify the cycles of $\sigma_0$ with the vertices $\vertices(\Gamma)$,
    the cycles of $\sigma_1$ with the edges $\edges(\Gamma)$,
    and the cycles of $\sigma_2$ with the disjoint union of the faces $\faces(\Gamma)$ and boundaries $\boundaries(\Gamma)$.
    Boundary cycles are identified by having their half-edges in $B$,
    while face cycles have their half-edges in $F = \halfedges \setminus B$.
        
    For a set $X$, a ribbon graph $\Gamma = (\halfedges, \sigma_0, \sigma_1, B)$ is called $X$-colored if there is a map from the set of faces of $\Gamma$ to $X$.
    Equivalently, this is a map
    $c: F \rightarrow X$ satisfying
    $c\circ\sigma_2 = c$.
    When $X = \{1, \ldots, n\}$, we call the graph $n$-colored.
    A \emph{face-marked} ribbon graph is an $n$-colored ribbon graph with a bijective coloring.

A \emph{metric} ribbon graph has the extra data of a positive length assigned to each edge. Hence, any point in $\RR_+^{\abs{\edges(\Gamma)}}$ defines a metric on $\Gamma$. However, the automorphism group of $\Gamma$ acts on
$\RR_+^{\abs{\edges(\Gamma)}}$, making the set of all metrics on $\Gamma$ into an orbifold:
\begin{equation*}
    \metrics(\Gamma) = \RR_+^{\abs{\edges(\Gamma)}} / \aut(\Gamma).
\end{equation*}

If we let $\graphs_{(g,b), n}$ be the set of all connected, face-marked ribbon graphs of type $((g,b), n)$, then the \emph{open ribbon graph complex} is the set
\begin{equation*}
    \openRG_{(g, b), n} = \bigsqcup_{\Gamma \in\graphs_{(g, b), n}} \metrics(\Gamma).
\end{equation*}
Edge contraction provides a natural topology on the complex, making it into a smooth orbifold \cite{MR1734132, sleator1988rotation}.
Note that $\openRG_{(g, 0), n} = \RG_{g,n}$ is the complex of closed ribbon graphs appearing previously in the literature.

In order for the constructions below to make sense, for any ribbon graph we must first choose a labeling on the boundary cycles.
This produces a $b!$-fold orbifold covering
\begin{equation*}
    \CoverOpenRG_{(g,b), n} \rightarrow \openRG_{(g,b), n}.
\end{equation*}
However, the end result is independent of choice of labeling, hence can be understood as a statement about $\openRG_{(g, b), n}$.

Following Kontsevich \cite{MR1171758}, we construct $n+b$ 2-forms on the ribbon graph complex
(one for each face and boundary)
by choosing a distinguished edge for every face and boundary cycle,
which upgrades the cyclic ordering of the edges around a cycle to a total ordering.
We use the notation $\ell_1^{[k]}, \ldots, \ell_{m_k}^{[k]}$ to indicate the lengths of the edges appearing in order around cycle $k$.
The face cycles have indices $k \leq n$, while the boundary cycles correspond with $k > n$.
We define
\begin{equation*}
 \omega_k = \sum_{i=1}^{m_k - 1}\sum_{j=i+1}^{m_k}d\ell_i^{[k]}\wedge d\ell_j^{[k]},
\end{equation*}
then set
\begin{equation*}
\Omega = \frac{1}{2} \sum_{k=1}^{n+b} \omega_k.
\end{equation*}
  
Note that $\Omega$ is not invariant under changes in the choices of total ordering at each boundary.
However, if $\Omega'$ is a Kontsevich 2-form constructed from a different choice of total ordering at each cycle,
then the difference can be calculated to be
\begin{equation*}
\Omega - \Omega' = \sum_{k=1}^{n+b} a_k \wedge dp_k,
\end{equation*}
where $a_k$ are one forms and $p_k = \ell_1^{[k]} + \cdots + \ell_{m_k}^{[k]}$ is the perimeter of cycle $k$.
If we let $p: \CoverOpenRG_{(g,b), n} \rightarrow \RR_+^{n+b}$ be the map sending a metric ribbon graph to its tuple of cycle perimeters, then we denote $\CoverOpenRG_{(g, b), n}(\vec{x}; \vec{y}) = p^{-1}(x_1, \ldots, x_n, y_1, \ldots, y_b)$.
Hence
$\left. \Omega\right|_{\CoverOpenRG_{(g,b), n}(\vec{x}; \vec{y})}$
is well-defined.
Moreover, Kontsevich \cite{MR1171758} proved that it is non-degenerate when restricted to cells corresponding to graphs with no even-valent vertices,
which includes the top-dimensional strata of trivalent ribbon graphs
(graphs where all vertices have degree 3).
He also proved that for any trivalent graph $\Gamma$ of type $((g, b), n)$, we have
 \begin{equation*}
     \frac{\Omega^d}{d!} \prod_{i=1}^n dx_i \prod_{j=1}^b dy_j
     = 2^{\alpha}
       \prod_{e\in \edges(\Gamma)} dl_e,
 \end{equation*}
 where $\alpha = 2g - 2 + n + b =  {\abs{\edges(\Gamma)} - \abs{\vertices(\Gamma)}}$, and
 $d = 3g-3 + n + b$.
 
We define
\begin{equation*}
\openVol_{(g, b),n}(\vec{x}; \vec{y}) =
\frac{1}{b!}
\int\limits_{\CoverOpenRG_{(g,b),n}(\vec{x}; \vec{y})} \frac{1}{d!}\Omega^d.
\end{equation*}
If we treat the boundary perimeters as flat coordinates on the moduli of open surfaces (which we justify in Section~\ref{sect:IntersectionNumbers}),
and take the Laplace transform of the face perimeter coordinates, we get the expression
\begin{align*}
    W_{(g, b), n}(\vec{\lambda}) &=
    \int_0^{\infty} \cdots \int_0^{\infty} d\vec{x} d\vec{y} \exp(- \vec{\lambda}\cdot\vec{x})
    \openVol_{(g,b),n}(\vec{x}; \vec{y}) \\
    &= 2^{\alpha} \int\limits_{\openRG_{(g, b), n}}\hspace{-3mm} \exp( -\vec{\lambda}\cdot \vec{x} ) 
    \prod_{e\in \edges(\Gamma)} d\ell_e. \\
\end{align*}
Naively, one might expect the integral to diverge, as there is no exponential dampening on the boundary perimeters.
However, condition (5) of Definition~\ref{defn:RibbonGraph} ensures that
\begin{equation*}
    \sum y_j \leq \sum x_i,
\end{equation*}
which allows the integral to converge.

We observe that for any face-marked ribbon graph $\Gamma$
\begin{equation*}
    \vec{\lambda}\cdot \vec{x} = \sum_{e \in \edges(\Gamma)} \tilde\lambda_e \ell_e,
\end{equation*}
where
\begin{equation*}
    \tilde\lambda_e = 
    \begin{cases}
        \lambda_{c(e_+)} + \lambda_{c(e_-)} & \text{if $e$ is an internal edge,} \\
        \lambda_{c(e)} & \text{if $e$ is a boundary edge.}
    \end{cases}
\end{equation*}
Note that $c(e_{\pm})$ refers to the two face colors on either side of an internal edge, while any boundary edge
has exactly one side being a face, making $c(e)$ unambiguous.

As well, the integral over $\openRG_{(g, b), n}$ defining $W_{(g, b), n}$ splits as a sum of integrals over the open strata, each cell of which corresponds to a graph in $\graphs^3_{(g, b), n}$, the set of all face-marked trivalent ribbon graphs of type $((g, b), n)$. Hence, we calculate
\begin{align}
    \nonumber
    W_{(g, b), n}(\vec{\lambda})
    &= \sum_{\Gamma \in \graphs_{(g,b), n}^3}
    \frac{2^{\abs{\edges(\Gamma)} - \abs{\vertices(\Gamma)}}}{\abs{\aut(\Gamma)}}
    \prod_{e\in \edges(\Gamma)}\int_{0}^{\infty} e^{-\tilde\lambda_e \ell_e} d\ell_e \\
    \label{eqn:W-GraphSum}
    &= \sum_{\Gamma \in \graphs_{(g,b), n}^3}
    \frac{2^{\abs{\edges(\Gamma)} - \abs{\vertices(\Gamma)}}}{\abs{\aut(\Gamma)}}
    \prod_{e\in \edges(\Gamma)} \frac{1}{\tilde\lambda_e}.
\end{align}

In Section~\ref{sect:FeynmanGraphs}, this identical sum appears as the Feynman graph expansion of the Kontsevich-Penner model,
while in Section~\ref{sect:IntersectionNumbers}, we justify the integral as a cohomological quantity on a particular compactification of the moduli space of open Riemann surfaces.

\section{Feynman graph expansion of the Kontsevich-Penner model}
\label{sect:FeynmanGraphs}

We wish to use Feynman graph techniques (c.f. \cite{bessis1980quantum}) to calculate the asymptotic expansion of the Kontsevich-Penner model
\begin{equation*}
    \tau_Q = \det(\Lambda)^Q \mathcal{C}^{-1}_{\Lambda}
    \int_{\mathcal{H}_N} dX \exp
    \Bigl(-\Tr\bigl(
            \frac{X^3}{3!}
            + \frac{X^2 \Lambda}{2}
            + Q\log(X+\Lambda)
        \bigr)
    \Bigr),
\end{equation*}
where
\begin{equation*}
    \mathcal{C}_{\Lambda} = e^{\Tr \Lambda^3 / 3} \int_{\mathcal{H}_N} dX\,
        e^{-\Tr \frac{X^2\Lambda}{2}},
\end{equation*}
$\Lambda = \diag(\lambda_1, \ldots, \lambda_N)$, and we are considering an expansion when $\Lambda \rightarrow (\infty, \ldots, \infty)$. The techniques involved are standard, and well described in the literature. Since the asymptotic expansion in this case is a variation of the expansion for the Kontsevich model, we follow the approach taken by Looijenga \cite{looijenga1992intersection} in what follows.

Because we are in the large $\Lambda$ regime, we can first expand
\begin{equation*}
    \exp\bigl(-\Tr Q\log(X+\Lambda)\bigr)
    = \det(\Lambda)^{-Q} \exp \sum_{k=1}^{\infty} \frac{Q}{k}\Tr(\Lambda^{-1}X)^k.
\end{equation*}

If we denote
\begin{equation*}
    \braket{f}_{\Lambda} = \mathcal{C}^{-1}_{\Lambda}
    \int_{\mathcal{H}_N} dX f(X) \exp\bigl(-\Tr \frac{X^2\Lambda}{2}\bigr),
\end{equation*}
then we wish to calculate
\begin{equation*}
    \braket{\exp\Tr\Bigl(-\frac{X^3}{3!} + \sum_{k=1}^{\infty} \frac{Q}{k}(-\Lambda^{-1}X)^k\Bigr)}_{\Lambda}.
\end{equation*}
To that end, we introduce
$\interiorHalfedges = \zset{d} \times \ZZ_3$,
$\boundaryHalfedges = \bigcup_{j=1}^{K} \zset{b_j} \times \ZZ_j$,
$\halfedges = \interiorHalfedges \cup \boundaryHalfedges$,
and the cyclic rotation operator
\begin{align*}
    \sigma_0 : \halfedges & \rightarrow \halfedges \\
    (i, j) & \mapsto (i, j+1).
\end{align*}
Note that an expression of the form
\begin{equation*}
    M(d; b_1, \ldots, b_K) = \bigl(\Tr(-X^3)\bigr)^d
    \prod_{j=1}^{K} \bigl(\Tr(-\Lambda^{-1}X)^j\bigr)^{b_j}
\end{equation*}
can be written out as a sum of monomials in the variables $X_{ij}$ and $\lambda_i^{-1}$, naturally labelled by the set of maps
\begin{equation*}
    \varphi : \halfedges \rightarrow \zset{N}
\end{equation*}
by the following correspondence. 
For a map $\varphi: \halfedges \rightarrow \zset{N}$, and
$a \in \interiorHalfedges$ we denote
$\varphi_a = X_{\varphi(a),\varphi\sigma_0(a)}$, while for
$b\in \boundaryHalfedges$ we have
$\varphi_b = \lambda^{-1}_{\varphi(b)} X_{\varphi(b),\varphi\sigma_0(b)}$. Then
\begin{equation}
    \label{eqn:MapSum}
    M(d; b_1, b_2, \ldots, b_K) = \sum_{\varphi: \halfedges \rightarrow \zset{N}}
    \prod_{c\in \halfedges} \varphi_c.
\end{equation}

By Wick's Lemma (c.f. \cite{bessis1980quantum}), we have
\begin{equation}
    \label{eqn:WickSum}
    \braket{\prod_{c\in\halfedges}\varphi_c}_{\Lambda} = 
    \sum_P \prod_{\{c, d\} \in P} \gamma(\varphi_c, \varphi_d),
\end{equation}
where 
\begin{equation*}
    \gamma(X_{i,j}, X_{k, l}) = \frac{2}{\lambda_i + \lambda_j} \delta_{i,l} \delta_{j,k},
\end{equation*}
and the sum is over the set of all pairings $P$ of elements of $\halfedges$.
Equivalently, this is the set of fixed-point-free involutions 
$\{\sigma_1: \halfedges \rightarrow \halfedges\ |\ \sigma_1^2 = 1, \sigma_1(x) \neq x\ \forall x \in \halfedges\}$,
where a pairing $P$ corresponds with an involution $\sigma_1$ by
$\{c, d\}\in P \iff \sigma_1(c) = d$.

As discussed in Section~\ref{sect:RibbonGraphs}, we can associate an unreduced ribbon graph
$\Gamma(\sigma_0, \sigma_1)$
to the pair of permutations, with the half-edges labelled bijectively by the elements of $\halfedges$.
We see that the graph has two distinct types of vertices:
    $d$ \emph{internal vertices}, each of degree 3,
    and $b_j$ \emph{boundary vertices} of degree $j$, for $j=1, \ldots, K$.
We further replace $\Gamma(\sigma_0, \sigma_1)$ with the reduced ribbon graph 
$\widetilde{\Gamma}(\sigma_0, \sigma_1)$,
obtained by ``blowing up'' each boundary vertex into a cycle of boundary edges, as depicted in Figure~\ref{fig:BoundaryBlowup}.
\begin{figure}
    \begin{tikzpicture}
        \draw (-0.75, -0.75) -- (0.75, 0.75);
        \draw (-0.75, 0.75) -- (0.75, -0.75);
        \filldraw (0, 0) circle (3pt);
        \draw [->] (1, 0) -- (2.5, 0);
        \draw (4, 0) circle (0.5)
            +(45:0.5) -- +(45:1.25)
            +(135:0.5) -- +(135:1.25)
            +(225:0.5) -- +(225:1.25)
            +(315:0.5) -- +(315:1.25);
        \draw[dashed] (4, 0) circle (0.4);
    \end{tikzpicture}
    \caption{Expanding a boundary vertex into a boundary cycle}
    \label{fig:BoundaryBlowup}
\end{figure}
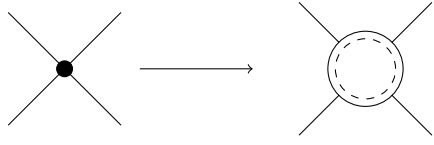
Note that $\widetilde{\Gamma}(\sigma_0, \sigma_1)$ is trivalent, and the set of boundary edges so obtained automatically satisfies condition (5) of Definition~\ref{defn:RibbonGraph}.
In addition, every half-edge $i$ incident to a boundary vertex of $\Gamma(\sigma_0, \sigma_1)$ contributes a weight of $\lambda_{\phi(i)}^{-1}$ to \eqref{eqn:WickSum}.
These weights can be uniquely transferred to the newly created boundary edges in
$\widetilde{\Gamma}(\sigma_0, \sigma_1)$.

Since $\gamma(X_{i,j}, X_{k,l})$ is non-zero exactly when $i=l$ and $j=k$, a graph
$\widetilde{\Gamma}(\sigma_0, \sigma_1)$
contributes to \eqref{eqn:WickSum} when the following condition is satisfied:
If $\sigma_1(c) = d$ then $\varphi(c) = \varphi\circ\sigma_0(d)$ and $\varphi(d) = \varphi\circ\sigma_0(c)$.  
Hence on the ribbon graph $\widetilde{\Gamma}$, $\varphi$ factors through an
$N$-coloring $\bar\varphi$ of its face cycles,
giving an $N$-colored ribbon graph $(\widetilde{\Gamma}, \bar\varphi)$.

Let $G$ denote the group of automorphisms of $\halfedges$ that commute with $\sigma_0$.
In other words, $\psi \in G$ if and only if
$\psi: \halfedges \rightarrow \halfedges$ with
$\psi(\interiorHalfedges) = \interiorHalfedges$,
$\psi(\boundaryHalfedges) = \boundaryHalfedges$, and
$\psi\circ\sigma_0 = \sigma_0\circ\psi$.
It is easy to see that $G$ is a direct product of semi-direct products, with
\begin{equation*}
    G = \left(S_d \cdot (\ZZ_3)^d\right) \times
    \prod_{j=1}^{K}\left( S_{b_j} \cdot (\ZZ_j)^{b_j}\right),
\end{equation*}
and its order given by
\begin{equation*}
    \abs{G} = d!3^d \prod_{j=1}^{K} b_j! j^{b_j}.
\end{equation*}

$G$ acts on the set of pairs $(\sigma_1, \bar\varphi)$, with two pairs defining isomorphic colored open ribbon graphs if and only if they are in the same $G$-orbit. Furthermore, the automorphism group of $(\widetilde{\Gamma}(\sigma_0, \sigma_1), \bar\varphi)$ is the $G$-stabilizer of $(\sigma_1, \bar\varphi)$.

Since
\begin{align}
    \tau_Q &=
    \braket{\exp\Tr\Bigl(
        -\frac{X^3}{3!} + \sum_{k=1}^{\infty} \frac{Q}{k}(-\Lambda^{-1}X)^k
    \Bigr)}_{\Lambda} \nonumber \\
    \label{eqn:TauExpansion}
    &= \sum_{d, (b_1, b_2, \ldots)} \frac{(-1)^d}{d! 3^d 2^d}
    \left(
        \prod_j \frac{Q^{b_k}}{b_j! k^{b_j}}
    \right)
    \braket{\Tr^d X^3 \prod \Tr^{b_j}(-\Lambda^{-1}X)^j}_{\Lambda},
\end{align}
we have
\begin{proposition}
\begin{equation}
    \label{eqn:TauGraphSum}
    \tau_Q = \sum_{(\widetilde{\Gamma}, \bar\varphi)}
    \frac{2^{\abs{\edges(\widetilde{\Gamma})} - \abs{\vertices(\widetilde{\Gamma})}}
          Q^{\abs{\boundaries(\widetilde{\Gamma})}}}
         {\abs{\aut(\widetilde{\Gamma}, \bar\varphi)}}
    \prod_{e\in\edges(\widetilde{\Gamma})} \tilde\lambda_e^{-1},
\end{equation}
where the sum is over the set of all trivalent $N$-colored ribbon graphs (both open and closed, connected and disconnected),
and we recall that
\begin{equation*}
    \tilde\lambda_e = 
    \begin{cases}
        \lambda_{\bar\varphi(e_+)} + \lambda_{\bar\varphi(e_-)} & \text{if $e$ is an internal edge} \\
        \lambda_{\bar\varphi(e)} & \text{if $e$ is a boundary edge}.
    \end{cases}
\end{equation*}
\end{proposition}

\begin{proof}
    The only outstanding issue is the fate of the negative signs present in \eqref{eqn:TauExpansion}.
    We observe that for any graph, $3d + \sum j b_j = 2\abs{\text{internal edges}}$.
    Since $-1$ appears in \eqref{eqn:TauExpansion} with exponent $d + \sum j b_j$, they completely cancel for \eqref{eqn:TauGraphSum}.
    This completes the proof.
\end{proof}

As is typical in these types of counting problems,
$F_Q = \log \tau_Q$ is obtained by restricting the sum to connected ribbon graphs.
Given our expression \eqref{eqn:W-GraphSum} in the previous section for
$W_{(g, b), n}(\vec{\lambda})$, we can immediately see
\begin{corollary}
\begin{equation*}
    F_Q = \sum_{g, b, n} \frac{Q^b}{n!}
    \sum_{\phi: \zset{n} \rightarrow \zset{N} }
    W_{(g, b), n}(\lambda_{\phi(1)}, \ldots, \lambda_{\phi(n)}).
\end{equation*}
\end{corollary}

\begin{example}
    The coefficient of $t_3 = \frac{1}{3}\sum \lambda_i^{-3}$ for $F_Q$ has contributions from two graphs, as depicted in Figure~\ref{fig:t3Graphs}.
    \begin{figure}
        \begin{tikzpicture}[scale=0.65]
            \draw[xshift=-2cm] (2, 0) circle (1.5cm);
            \draw[rounded corners=15pt, xshift=-2cm] (3.5, 0.75) +(-90:0.75) arc
            (-90:127:0.75)
            (3.5, 0.75) +(160:0.75) -- (2, 0) -- (0.5, 0)
            (2, 0);
            \draw[xshift=6cm] (0, 0) circle (1)
                (1, 0) -- (3, 0)
                (4, 0) circle (1);
            \draw[xshift=6cm, dashed] (0, 0) circle (0.8)
                (4, 0) circle (0.8);
        \end{tikzpicture}
        \caption{Graphs contributing to the coefficient of $t_3$ in $F_Q$. Dashed lines delineate boundary cycles.}
        \label{fig:t3Graphs}
    \end{figure}
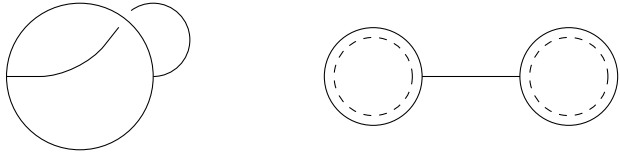
    The first graph has no boundary components and an automorphism group of order 6, while the second graph has two boundary components and an automorphism group of order 2. Hence
    \begin{equation*}
        [t_3]F_Q =
            \frac{1}{8} + \frac{3}{2}Q^2.
    \end{equation*}
\end{example}

\begin{example}
    The coefficient of $t_1 t_2 = \frac{1}{2} \sum \lambda_i^{-1}\sum\lambda_i^{-2}$ for $F_Q$ has contributions from the two graphs depicted in Figure~\ref{fig:t1t2Graphs}.
    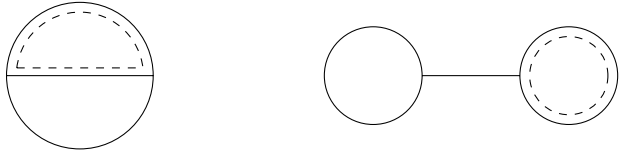
\begin{figure}
        \begin{tikzpicture}[scale=0.65]
            \draw (0, 0) circle (1.5)
                (-1.5, 0) -- (1.5, 0);
            \draw[dashed] (173:1.3) arc (173:7:1.3)
                (173:1.3) -- (7:1.3);
            \draw[xshift=6cm] (0, 0) circle (1)
                (1, 0) -- (3, 0)
                (4, 0) circle (1);
            \draw[xshift=6cm, dashed] (4, 0) circle (0.8);
        \end{tikzpicture}
        
        \caption{Graphs contributing to the coefficient of $t_1 t_2$ in $F_Q$.}
        \label{fig:t1t2Graphs}
    \end{figure}
    All graphs have a single boundary component and trivial automorphism group, except for the left graph in the case when the two faces have identical colors, where the automorphism group has order 2. Hence
    \begin{equation*}
        [t_1t_2]F_Q = 2Q.
    \end{equation*}
\end{example}

\section{Combinatorial representations of surfaces with boundary}
\label{sect:CombinatorialRepresentations}
In this section we prove that the complex of open ribbon graphs is equivalent to the moduli space of open Riemann surfaces. This equivalence enables the construction of a simple compactification.
In what follows, all surfaces have at least one interior marked point,
and we exclude the unstable surfaces of type
$((0, 0), 1)$, $((0, 0), 2)$, and $((0, 1), 1)$.

\begin{definition}
    A \emph{Riemann surface} of type $((g, b), n)$ consists of a complex analytic structure on a compact surface of genus $g$, with $b$ boundary components,
    together with $n$ distinct marked points $a_1, \ldots, a_n$ in the interior of the surface.
    The Riemann surface is \emph{closed} when $b=0$, otherwise it is \emph{open}.
    We require each boundary component to have a holomorphic collar structure.
    The moduli space of all such surfaces is denoted $\openModuli_{(g,b),n}$.
    It is a real analytic orbifold of dimension $6g - 6 + 3b  + 2n$.
\end{definition}

From the definition of a Riemann surface $\Sigma$ of type $((g, n), n)$,
one may canonically associate to such a surface a closed Riemann surface by constructing the double $\double\Sigma$.
If $\Sigma$ has genus $g$, with $b$ boundary components and $n$ interior marked points,
then $\double\Sigma$ will have genus $2g +b - 1$ and $2n$ interior marked points.
Moreover, it comes equipped with an anti-holomorphic involution
$\rho: \double\Sigma \rightarrow \double\Sigma$,
where the quotient space is equivalent to $\Sigma$, and the fixed point set is identified with the boundary of $\Sigma$.

We may use the uniformization theorem to uniquely associate to $\double\Sigma$ a complete, finite area hyperbolic metric on the punctured surface
($\double\Sigma$ with the marked points removed).
This conformally equivalent hyperbolic surface will have an orientation reversing isometric involution, with the fixed point set being a union of simple closed geodesics.

We wish to assign a closed metric ribbon graph to $\double\Sigma$, but to do so uniquely we must first specify a positive weight for each marked point.
In our case, we must choose weights so that if marked point $a_i$ has weight $x_i \in \RR_+$,
then $\rho(a_i)$ also has weight $x_i$.
One method of constructing a ribbon graph that captures the geometry of the surface is to use the complex structure through Jenkins-Strebel differentials (c.f. \cite{strebel1984quadratic}).
However, the current work is more naturally suited to the cut-locus construction of Bowditch and Epstein \cite{MR935529},
which uses the hyperbolic structure of the surface.

To summarize the approach, we first use the weights to find horocycles in a neighborhood of each marked point/puncture. 
In particular, we must uniformly rescale the weights so that they sum to 1 (the construction is scale invariant),
then choose the unique horocycle of length given by the rescaled weight of that puncture. We let $\widehat{\Gamma} \subset \double\Sigma$ be the set of points with two or more shortest geodesics to the collection of horocycles.
As proven in \cite{MR935529}, the set $\widehat{\Gamma}$ enjoys a number of nice properties, including:
\begin{enumerate}
    \item $\widehat{\Gamma}$ is a closed ribbon graph, with no vertices of degree 1 or 2. The faces of the ribbon graph are homotopic to the horocycle neighborhoods of the punctures in $\double\Sigma$.
    \item The edges of $\widehat{\Gamma}$ are geodesic segments in $\double\Sigma$.
    \item Each edge can be assigned a length by taking the length of the section of the horocycle
        corresponding to that edge (there is a symmetry exchanging the two sides of the edge, making it a well-defined quantity).

    \item The sum of the assigned edge lengths around a face cycle equals the weight of the corresponding puncture.
\end{enumerate}

For the case at hand, where the surface has an isometric involution, it is easy to see that the isometry preserves the ribbon graph. Moreover, the fixed point set of the involution is a subset of the ribbon graph. Hence the quotient graph is a ribbon graph, where the collection of boundary edges is exactly the set of edges fixed point-wise by the involution.
Thus after following the arguments presented in \cite{MR935529}, which adapt essentially without change to the present situation, we have
\begin{theorem}
    In the commutative diagram
\begin{equation*}
    \begin{tikzcd}[column sep=tiny]
        \openModuli_{(g,b), n} \times \RR_+^n \arrow{rr}{\Phi}
        \arrow{dr}
        & & \openRG_{(g, b), n} \arrow{dl} \\
        & \RR_+^n,&
    \end{tikzcd}
\end{equation*}
   the Bowditch-Epstein map $\Phi$ is an equivalence of orbifolds. 
\end{theorem}

We note that an open ribbon graph can be thought of as a closed ribbon graph via the forgetful map which does not distinguish between boundary and non-boundary edges. Hence we have the sequence of maps
\begin{equation*}
    \CoverOpenRG_{(g,b), n}
    \hookrightarrow \RG_{g, b+n} 
    \xrightarrow{\sim} \moduli_{g, b+n} \times \RR^{b+n}_+
    \hookrightarrow \Mbar_{g, b+n}\times \RR^{b+n}_{\geq 0},
\end{equation*}
where we recall that $\CoverOpenRG_{(g, b), n}$ is the $b!$-fold cover of $\openRG_{(g, b), n}$ obtained by labeling the $b$ boundary components.
Since all of the maps are equivariant with respect to the natural $S_b$ action, we have an inclusion
\begin{equation*}
    \beta: \openModuli_{(g, b), n} \times \RR_+^n \hookrightarrow
    (\Mbar_{g, b+n} \times \RR^{b+n}_{\geq 0})/S_b.
\end{equation*}

Next note that for any fixed
$\vec{x} = (x_1, \ldots, x_n)$, the closure of $\beta(\openModuli_{(g,b), n} \times \{\vec{x}\})$ is compact.
This follows from the fact that for any open metric ribbon graph
$\Gamma \in \openRG_{(g,b), n}$, having face perimeters of length $\vec{x}$ and boundary perimeters of length $\vec{y}$,
we must have the inequality
\begin{equation*}
    \sum_{i=1}^b y_i \leq \sum_{j=1}^n x_j.
\end{equation*}
We denote this compact space by $\compactOpenModuli_{(g, b), n} (\vec{x})$.
We will also make use of the compact space
$\compactOpenModuli_{(g, b), n}(\vec{x}; \vec{y}) \subset \Mbar_{g, b+n} \times \RR_{\geq 0}^{b+n}$,
which is the closure of the image of $\CoverOpenRG_{(g, b), n}(\vec{x}; \vec{y})$.

\section{Intersection number calculations}
\label{sect:IntersectionNumbers}

In this section we justify the integral of Kontsevich's piecewise-defined form over the ribbon graph complex as a tautological class calculation on the compactification of the moduli space of open Riemann surfaces.
To do so we adapt to the open Riemann surface case a scaling-limit procedure for the Weil-Petersson form given by Do \cite{2010arXiv1010.4126D}.

As is well known, the moduli space $\moduli_{g,n}$ has a symplectic structure given by the Weil-Petersson form
$\omega_{WP}$.
Moreover, Wolpert \cite{wolpert1983homology} has shown that the form extends smoothly to the boundary in $\Mbar_{g,n}$, defining a cohomological class $[\omega_{WP}] \in H^2(\Mbar_{g,n}; \QQ)$.
However, for our purposes, we also need the Weil-Petersson form defined on the moduli space of bordered hyperbolic surfaces,
$\moduli_{g,n}(L_1, \ldots, L_n)$,
where a point in this space is a hyperbolic metric on a compact surface with $n$ boundary components,
with the boundaries being geodesics of specified lengths $(L_1, \ldots, L_n)$.
Note that this space has more in common with the usual moduli space $\moduli_{g,n}$, being, in fact, diffeomorphic, rather than the moduli space of open Riemann surfaces under consideration in the present work.
It was proven by Mirzakhani \cite{Mirzakhani:2007kc} that the pull back of Weil-Petersson forms under a diffeomorphism
\begin{equation*}
    f_{\vec{L}} : \Mbar_{g, n} \rightarrow \Mbar_{g,n}(L_1, \ldots, L_n)
\end{equation*}
satisfies
\begin{equation*}
    f^{*}_{\vec{L}}[\omega_{WP}] = [\omega_{WP}] + \frac{1}{2}\sum L_i^2 \psi_i.
\end{equation*}

Furthermore, Mondello \cite{mondello2006triangulated} and Do \cite{2010arXiv1010.4126D} have proven that in the scaling limit
$\vec{L} \rightarrow \infty$,
the Weil-Petersson form converges pointwise to Kontsevich's symplectic form.
To be more precise, the Bowditch-Epstein construction works equally well on surfaces in
$\moduli_{g, n}(\vec{L})$,
with the weights coming from the lengths of the geodesic boundaries.
If 
\begin{align*}
    \Phi & : \moduli_{g,n} \times \RR_+^{n} \xrightarrow{\sim} \RG_{g,n} \\
    \Phi_{\vec{L}} & : \moduli_{g,n}(\vec{L}) \xrightarrow{\sim} \RG_{g,n}(\vec{L})
\end{align*}
are the two different Bowditch-Epstein diffeomorphisms, then we denote the composition
\begin{equation*}
    f_{\vec{L}} = \Phi_{\vec{L}}^{-1} \circ
    \Phi \Bigr|_{\moduli_{g,n}\times \{\vec{L} \} } :
    \moduli_{g,n} \xrightarrow{\sim} \moduli_{g,n}(\vec{L}).
\end{equation*}
Note that this map is an expression of the fact that for any metric ribbon graph, one can uniquely construct a cusped hyperbolic surface, and also a bordered hyperbolic surface.

We perform a scaling limit by considering the pullback of the rescaled Weil-Petersson form $\frac{1}{t^2}\omega_{WP}$ under the family of maps
$f_{t\vec{L}}$, as $t\rightarrow \infty$.
In fact, we have \cite{2010arXiv1010.4126D}
\begin{equation*}
    \Omega = \lim_{t\rightarrow\infty}\frac{1}{t^2} f^*_{t\vec{L}} \omega_{WP},
\end{equation*}
where the convergence is pointwise on the open dense subset of $\moduli_{g, n}$ corresponding with trivalent ribbon graphs in the combinatorial model.

Do uses this fact to justify the equality
\begin{equation*}
    \int\limits_{\RG_{g,n}(\vec{L})} \frac{\Omega^d}{d!} = 
    \int\limits_{\Mbar_{g,n}} \frac{1}{d!}\Bigl(
        \sum \frac{L_i^2}{2}\psi_i
    \Bigr)^d,
\end{equation*}
thus sidestepping the delicate quotient compactification of $\RG_{g,n}$ in Kontsevich's original proof of the Witten conjecture.

The same analysis applies verbatim to the moduli space of open Riemann surfaces, considered through its image in $\moduli_{g, b+n}$.
As a result, one sees that
$\Omega \bigr|_{\CoverOpenRG_{(g,b), n}(\vec{x}; \vec{y})}$
represents the scaled sum of $\psi$-classes
$\sum \frac{x_i^2}{2}\psi_i + \sum \frac{y_j^2}{2}\psi_{j+n}$.
In other words, we have proven
\begin{theorem}
    \begin{equation*}
        \openVol_{(g, b), n}(\vec{x}; \vec{y}) = 
        \frac{1}{b!} \int\limits_{\compactOpenModuli_{(g, b), n}(\vec{x}; \vec{y})}
        \frac{1}{d!} \biggl( 
            \sum \frac{x_i^2}{2}\psi_i + \sum \frac{y_j^2}{2}\psi_{j+n}
        \biggr)^d,
    \end{equation*}
    where $d=3g-3 + n + b$.
\end{theorem}

The one remaining subtle point concerns justifying integrating over the $y$-variables as flat coordinates. The explanation is as follows.

We let
\begin{equation*}
\begin{tikzcd}
    (S^1)^b \arrow[hook]{r} &  \torusOpenModuli_{(g, b), n} \arrow{d} \\
    & \openModuli_{(g,b), n}
\end{tikzcd}
\end{equation*}
be the $b$-torus orbifold bundle obtained by allowing exactly one (labeled) marked point on each boundary component.
In this notation, we understand $\vec{\gamma} = (\gamma_1, \ldots, \gamma_b)$ to be the boundary geodesics of the surface.
This space has an equivalent representation
\begin{equation*}
\begin{tikzcd}
    \torusOpenModuli_{(g,b), n} \ar{r}{\sim} & \moduli^{\vec{\gamma}}_{g, n+2b} \arrow{d} \\
    & \moduli_{g, n+2b}
\end{tikzcd}
\end{equation*}
obtained by capping off each boundary component with a sphere with 1 boundary and 2 cusps, as depicted in Figure~\ref{fig:CappingSurface}.
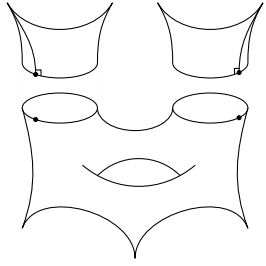
\begin{figure}
\begin{center}
\begin{tikzpicture}[scale=1]
    \draw (0,0) ellipse (0.5 and 0.20)
        (2, 0) ellipse (0.5 and 0.2)
        (-.5, .60) arc (-180: 0: 0.5 and 0.2)
        (1.5, .6) arc (-180: 0: 0.5 and 0.2);
    \draw (0.5, 0) arc (-180:0: 0.5 and 0.3)
        (-0.5, 0) to[out=-75, in=70] (-.5, -1.6)
        (-.5, -1.6) to[out=70, in=90] (1, -2.0)
        (1, -2.0) to[out=90, in=110] (2.5, -1.6)
        (2.5, 0) to[out=-105, in=110] (2.5, -1.6);
    \draw (-.5, .6) to[out=90, in=-60] (-.7, 1.4)
        (-.7, 1.4) to[out=-60, in=-120] (.7, 1.4)
        (.7, 1.4) to[out=-120, in=90] (.5, .6);
    \draw[xshift=2cm] (-.5, .6) to[out=90, in=-60] (-.7, 1.4)
        (-.7, 1.4) to[out=-60, in=-120] (.7, 1.4)
        (.7, 1.4) to[out=-120, in=90] (.5, .6);
    \filldraw (0, 0) ++(-130: 0.5 and 0.2) circle (0.75pt);
    \filldraw (2, 0) ++(-40: 0.5 and 0.2) circle (0.75pt);
    \draw (0, .6) ++(-130: 0.5 and 0.2) to[out=90, in=-60] (-.7, 1.4);
    \draw (2, .6) ++(-40: 0.5 and 0.2) to[out=90, in=-120] (2.7, 1.4);
    \filldraw (0, .6) ++(-130: 0.5 and 0.2) circle (0.75pt);
    \draw (0, .6) ++(-120: 0.5 and 0.2) --  ++(0, 0.08) -- ++(-0.08, 0);
    \draw (2, .6) ++(-50: 0.5 and 0.2) -- ++(0, 0.08) -- ++(0.07, 0);
    \filldraw (2, .6) ++(-40: 0.5 and 0.2) circle (0.75pt);
    \draw[yshift=0.4mm] (0.3, -.8) to[out=-40, in=-140] (1.8, -.8);
    \draw (0.5, -.9) to[out=45, in=135] (1.6, -.9);
    \draw[dash pattern=on 0pt off 2\pgflinewidth, yshift=0.6cm] (-130: 0.5 and 0.2) -- ++(0, -0.6)
        (180: 0.5 and 0.2) -- ++(0, -0.6)
        (0: 0.5 and 0.2) -- ++(0, -0.6);
    \draw[dash pattern=on 0pt off 3\pgflinewidth, xshift=2cm, yshift=0.6cm] (-40: 0.5 and 0.2) -- ++(0, -0.6)
        (180: 0.5 and 0.2) -- ++(0, -0.6)
        (0: 0.5 and 0.2) -- ++(0, -0.6);
\end{tikzpicture}
\end{center}
\caption{Capping an open surface.}
\label{fig:CappingSurface}
\end{figure}
Note that the unique geodesic seam joining the boundary of a capping sphere to one of its cusps is lined up with the marked point on the boundary.
This results in a point in the space
\begin{equation*}
    \moduli^{\vec{\gamma}}_{g, n+2b} = 
    \Biggl\{ 
        (C, \gamma_1, \ldots, \gamma_b) \  \Biggl|\ 
        \parbox{7.5cm}{$C \in \moduli_{g, n+2b}$, with marked cusps $a_1, \ldots, a_{n+2b}$,
        and $\gamma_i$ is a simple closed geodesic enclosing marked points $a_{n+2i+1}$, $a_{n+2i+2}$}
    \Biggr\},
\end{equation*}
which is an infinite cover of $\moduli_{g, n+2b}$.
To say that $\gamma_i$ encloses marked points $a_{n+2i+1}$ and $a_{n+2i+2}$ means that $C \setminus \gamma_i$ is the disjoint union of two surfaces, one of which is homeomorphic to a sphere with one boundary and two cusps, the cusps being labeled by $a_{n+2i+1}$ and $a_{n+2i+2}$.
This type of cover was first considered by Mirzakhani \cite{Mirzakhani:2007zt, Mirzakhani:2007kc} when calculating the Weil-Petersson volumes of the moduli space of bordered hyperbolic surfaces.
Some important observations about $\moduli^{\vec{\gamma}}_{g, n+2b}$ include the following:
\begin{itemize}
    \item It is a symplectic manifold (via the Weil-Petersson form).

    \item It has a hamiltonian $b$-torus action given by performing Fenchel-Nielsen twists (c.f. \cite{MR590044}) around the curves $\gamma_1, \ldots, \gamma_b$.
        This corresponds with rotating the marked points on $\torusOpenModuli_{(g,b), n}$.

    \item The moment map is given by the $b$-tuple of squares of lengths of the curves $\vec{\gamma}$, i.e. 
        $\mu(C, \vec{\gamma}) = \bigl(\ell^2(\gamma_1)/2, \ldots, \ell^2(\gamma_b)/2 \bigr)$,
        or equivalently the squared lengths of the geodesic boundaries in
        $\torusOpenModuli_{(g, b), n}$.

    \item The symplectic quotient at level set $\vec{y}$, i.e. 
        $\mu^{-1}(\vec{y}) / (S^1)^b$, is the space $\coverOpenModuli_{(g, b), n}(\vec{y})$, which consists of open surfaces where the boundaries are labeled and of specified lengths $y_1, \ldots, y_b$.

    \item The scaling limit argument for the Weil-Petersson form can be easily adapted to these symplectic quotients.
        Hence the form $\openVol_{(g, b), n}(\vec{x}; \vec{y})dy_1\cdots dy_b$ is the Duistermaat-Heckman \cite{MR674406} measure for the above Hamiltonian torus action.
\end{itemize}
Putting this all together, we have proven
\begin{theorem}
    $W_{(g,b), n}(\vec{\lambda})$ is the Laplace transform (w.r.t. variables $\vec{x}$) of the Duistermaat-Heckman volume obtained by integrating the form
    \begin{equation*}
        \frac{1}{d!}\left(
            \sum_{i=1}^{n} \frac{x_i^2}{2}\psi_i + \sum_{j=1}^{b}
            \frac{y_j^2}{2}\psi_{j+n}
        \right)^d d\vec{y}
    \end{equation*}
    over $\compactOpenModuli_{(g, b), b}(\vec{x})\subset \Mbar_{g,b+n} / S_b$.
\end{theorem}

This completes the argument that our integral formulas are geometrically relevant, and correspond with an intersection theory calculation on the moduli space of open Riemann surfaces.


\bibliographystyle{plain}
\bibliography{references}

\end{document}